\documentclass[11pt]{amsart}
\usepackage{amsmath,amsthm}
\usepackage{amssymb,latexsym}
\usepackage{enumerate}

\newtheorem{thm}{Theorem} %[section]
\newtheorem{lem}{Lemma}

\theoremstyle{definition}
\newtheorem{rem}{Remark}

\numberwithin{equation}{section}
\makeatletter
\@namedef{subjclassname@2010}{%
  \textup{2010} Mathematics Subject Classification}
\makeatother

\newcommand{\Leg}[2]{\left(\frac{#1}{#2}\right)}

\newcommand{\Z}{\mathbb{Z}}

\newcommand{\Oc}{\mathcal{O}}

\begin{document}

%%%%%%%%%%%%%%%%

\title[Explicit upper bound]{Explicit upper bound for an average number of divisors of quadratic polynomials}

\author[K. Lapkova]{Kostadinka Lapkova}
\address{MTA Alfr\'ed R\'enyi Institute of Mathematics\\
1053 Budapest, Re\'altanoda u. 13-15, Hungary}
\email{lapkova.kostadinka@renyi.mta.hu}

\date{17.10.2015}

\begin{abstract} Consider the divisor sum $\sum_{n\leq N}\tau(n^2+2bn+c)$ for integers $b$ and $c$ which satisfy certain extra conditions. For this average sum we obtain an explicit upper bound, which is close to the optimal. As an application we improve the maximal possible number of $D(-1)$-quadruples.
\end{abstract}

% 11N56 Rate of growth of arithmetic functions
%11A25 Arithmetic functions; related numbers; inversion formulas
%11D09 Quadratic and bilinear equations

\maketitle

\section{Introduction}\label{sec:intro}
%
%There are a lot of divisor estimates.
%Some questions on Sum d(f(n)), f a polynomial are very hard.
%Some problmes have been stlled asymptotically by Hooley.
%Ertdos had some results, see also the Elsholtz Tao paper,
%http://arxiv.org/abs/1107.1010
%section 7,  with some results.
%
%For f a cublic polynomial or higher degree polynomial no asymptotic formula is known.
%
%For some applications explicit estimates are required, rather than asymptotic bounds. There are not so many *explicit* divisor esytimates in the literature.
%(you can refer to Bordelles' paper (mentioned in my paper with Filipin and Fujita), my paper and papers by Trudgian.
%
%and then mention the application that your main theorem, as as side effect improves the diophantine quadruples.
%

\hspace{0.5cm} Let $\tau(n)$ denote the number of positive divisors of the integer $n$ and $P(x)\in\Z[x]$ be a polynomial. There is a lot of research on estimating average sums of divisors 
\begin{equation}\label{S1}\sum_{n=1}^N\tau\left(P(n)\right)\,.
\end{equation}
One of the ground-laying results was obtained by Erd\H os \cite{erdos}, who showed that for an irreducible polynomial $P(x)\in\Z[x]$ and for any $N>1$, we have 
$$N \log N \ll_P\sum_{n=1}^N \tau(P(n)) \ll_P N \log N ,$$
where the dependence in the constants can be both on the degree and the coefficients of the polynomial $P(x)$. While for quadratic polynomials there are asymptotic formulas for the sum (\ref{S1}) , e.g. in works of Hooley \cite{hooley}, McKee \cite{mckee}, \cite{mckee2}, and most recently in the paper of Dudek \cite{dudek}, the case $\deg P(x)\geq 3$ is much harder, and no asymptotic formulas for (\ref{S1}) are known in this case. A certain progress in this direction was made by Elsholtz and Tao in \S7 of \cite{elsh-tao}. \\

For some applications one needs explicit upper bounds for sum of divisors, rather than asymptotic formulas. Such explicit upper bounds for quadratic polynomials are scarce in the literature, and not always close to the optimal, i.e. with a main term of the same order of magnitude as the main term in the asymptotic formula. For example, for the polynomial $P(n)=n^2+1$ one can apply the theorem of McKee \cite{mckee} and obtain
\begin{equation}\label{asymp}\sum_{n=1}^N\tau(n^2+1)=\frac 3 \pi N\log N+\Oc(N)\sim 0.955\cdot N\log N\,.
\end{equation}
For this polynomial in Lemma 3.7 of \cite{elsh} Elsholtz, Filipin and Fujita give the explicit bound
\begin{equation}\label{Elsh_ineq}\sum_{n=1}^N\tau(n^2+1)<N\left((\log N)^2+4 \log N+2\right)\,,
\end{equation}
which is clearly larger by a factor of logarithm from the expected growth. This explicit upper bound was improved by Trudgian in \cite{trudgian}, but still with a main term of magnitude $N(\log N)^2$. \\

In this note we present an explicit upper bound for (\ref{S1}) for a family of quadratic polynomials, which includes the polynomial $P(n)=n^2+1$ as well. Our bound will be of the right order of magnitude $N\log N$, as predicted by the asymptotic formulas. The reason for considering only polynomials $P(n)=an^2+bn+c$ with $a=1$ and even integer $b$ is the main role of a certain Dirichlet convolution described in Lemma \ref{lem1} below. Here is our main result.

%------------------------

\begin{thm}\label{thm1}Let $f(n)=n^2+2bn+c$ for integers $b$ and $c$, such that the discriminant $\delta:=b^2-c$ is non-zero and square-free, and $\delta\not\equiv 1\pmod 4$.  Assume also that for $n\geq 1$ the function $f(n)$ is positive and non-decreasing. Then for any integer $N\geq 1$ there exist   positive constants $C_1$, $C_2$ and $C_3$, such that
$$\sum_{n=1}^N \tau(n^2+2bn+c)<C_1 N\log N + C_2 N  + C_3.$$
Let $\xi=\sqrt{1+2|b|+|c|}$, $A$ be the least positive integer such that $A\geq\max\left(|b|,|c|^{1/2}\right)$ and 
$\varkappa=g(4|\delta|)$ for $g(q)=\frac 1 2\sqrt{q}\log{q}+1.2\sqrt{q}$. Then we have 
\begin{eqnarray}\label{defC}C_1&=&1.216(\log \varkappa+2)\,,\nonumber\\
C_2&=&2\left(\varkappa+(\log\varkappa+2)(0.608\cdot\log\xi+1.166)\right)\,,\\
C_3&=&2\varkappa A\,.\nonumber
\end{eqnarray}

%\[\varkappa=\varkappa(\delta)=\left\{\begin{array}{ll}
%g(|\delta|)& \text{, if } \delta\equiv 0\pmod 2\, ;\\
%g(4|\delta|)&\text{, if } \delta\equiv 2,3\pmod 4\,,
%\end{array}\right.
%\]
%for $g(q)=\frac 1 2\sqrt{q}\log{q}+1.2\sqrt{q}$. Then for $|\delta|\neq 2$ we have 
%\begin{eqnarray}\label{defC}C_1&=&1.216(\log \varkappa+2)\,,\nonumber\\
%C_2&=&2\left(\varkappa+(\log\varkappa+2)(0.608\cdot\log\xi+1.166)\right)\,,\\
%C_3&=&2\varkappa A\,.\nonumber
%\end{eqnarray}
%If $|\delta|=2$, we take
%\begin{eqnarray}\label{defC2}C_1&=&1.216\cdot\varkappa\,,\nonumber\\
%C_2&=&2\varkappa(0.608\cdot\log\xi+2.166)\,,\\
%C_3&=&2\varkappa A\,.\nonumber
%\end{eqnarray}

\end{thm}
%-----------------------------------

\begin{rem}The constant $\varkappa$ comes from an effective P\'olya-Vinigradov inequality for a real Dirichlet character. We introduce the quantities $\xi$ and $A$, such that always when $n\geq 1$, we have $\sqrt{f(n)}\leq\xi n$ and $\sqrt{f(n)}\leq n+A$. 
\end{rem}
When we know the precise form of the quadratic polynomial and the corresponding character, we might achieve better upper bounds than the ones provided in Theorem \ref{thm1}. This is the case for the polynomial $f(n)=n^2+1$.%, when both $\varkappa$ and $\xi$ in (\ref{defC}) can be replaced by the number $1$, and $C_3$ by $0$.
\begin{thm}\label{thm2}For any integer $N\geq 1$ we have 
$$\sum_{n=1}^N \tau(n^2+1)<1.216\cdot N\log N+4.332\cdot N .$$
\end{thm}

We can give an application of Theorem \ref{thm2}. Define a \textit{$D(n)-m$-tuple} for a nonzero integer $n$ and a positive integer $m$ to be a set of $m$ integers such that the product of any two of them increased by $n$ is a perfect square. In the paper of Elsholtz, Filipin and Fujita \cite{elsh} a crucial role for bounding from above the possible number of $D(-1)$-quadruples plays the inequality (\ref{Elsh_ineq}). Plugging the result of Theorem \ref{thm2} in the proof of Theorem 1.3 \cite{elsh} from the paper of Elsholtz et al.  we obtain 

\begin{thm} There are not more than $4.7\cdot 10^{58}$ $D(-1)$-quadruples.
\end{thm}

This improves the upper bounds $4\cdot 10^{70}$ from \cite{cipu}, $5\cdot 10^{60}$ from \cite{elsh} and $3.01\cdot 10^{60}$ from \cite{trudgian} for the maximal possible number of $D(-1)$-quadruples, whereas it is conjectured there are none. Note, however, that even if we could supply constants closer to the ones in the asymptotic formula (\ref{asymp}), we could not achieve essentially useful upper bound for the maximal possible number of $D(-1)$-quadruples without any new ideas. This is due to the method used in the proof of Theorem 1.3 \cite{elsh} and the central role of the variable $N\sim 10^{55}$. 

%----------------------------------------------------------------------------------------------------------------
%----------------------------------------------------------------------------------------------------------------

\section{Proof of Theorem \ref{thm1}}\label{proof}
% Let us write $f(n):=n^2+2bn+c$. 
Since $\delta\neq 0$, the polynomial $f(n)$ is not a full square. It also represents positive non-decreasing function, therefore we can apply the Dirichlet hyperbola method :
$$\sum_{n\leq N}\tau(n^2+2bn+c)=\sum_{n\leq N}\sum_{d|f(n)}1=2\sum_{n\leq N}\sum_{\substack{d\leq \sqrt{f(N)}\\ d|f(n)}}1=2\sum_{d\leq\sqrt{f(N)}}\sum_{\substack{n\leq N\\ f(n)\equiv 0(d)}}1\,.$$
 Let 
\begin{equation}\label{def1} \rho(d):=\#\left\{0\leq m<d : m^2+2bm+c\equiv 0\pmod d \right\}\,.
\end{equation}
Then for the innermost sum we have
$$\sum_{\substack{n\leq N\\ f(n)\equiv 0(d)}}1\leq \left[\frac N d\right] \rho(d)+\rho(d)\leq \frac N d \rho(d)+\rho(d)\,,$$
so we obtain
\begin{equation}\label{eq1}\sum_{n\leq N}\tau(f(n))\leq 2N\sum_{d\leq \sqrt{f(N)}}\frac {\rho(d)} d + 2\sum_{d\leq \sqrt{f(N)}}\rho(d)\,.
\end{equation}
\\
\par We will bound the sums involving the function $\rho(d)$. For this a crucial role plays the presentation of $\rho(d)$ as a Dirichlet convolution of two other well-understood multiplicative functions. More precisely, consider the function $\mu^2(n)$, where $\mu$ is the M\"obius function, i.e. this is the square-free characteristic function. Also let $\chi(n)$ be the real Dirichlet character given by $\chi(1)=1$ and for $n\geq 1$
\begin{equation}\label{def2}
\chi(n)=\left\{\begin{array}{ll}
\Leg{\delta}{n}&\text{, if }(n,2\delta)=1 ;\\
0&\text{, otherwise ,}
\end{array}\right.
\end{equation}
where $\Leg{\delta}{n}$ is the Jacobi symbol.\\

%--------------------------------------------------
%--------- Lemma 1, Dirichlet convolution

The following lemma can be considered on the one hand as a generalization of an identity due to Hooley \cite{hooley1}, which he shows only for $b=0$. On the other hand, we work on a simplified case, with certain limitations on the discriminant $\delta$. Interestingly, in \cite{hooley} Hooley claims that with his methods he can give an asymptotic formula for the divisor sum (\ref{S1}) for a general quadratic polynomial $P(n)=an^2+bn+c$. Our guess is that he had in mind a similar Dirichlet series presentation as formula (8) in \cite{hooley}, but he never published this argument for the more general case. So, albeit not unexpected, our Lemma has not been published before. 
 
\begin{lem}\label{lem1} Let $\delta=b^2-c$ be square-free and $\delta\not\equiv 1\pmod 4$. Given the definitions (\ref{def1}) and (\ref{def2}), we have the identity
$$ \rho(d)=\sum_{lm=d}\mu^2(l)\chi(m)\,.$$
\end{lem}

\begin{proof}
First we notice that $\rho(1)=1, \rho(2)=1$ and $\rho(2^k)=0$ for $k\geq2$. Indeed, $n^2+2bn+c=(n+b)^2-b^2+c=(n+b)^2-\delta\,,$
so we have
$$\rho(d)=\#\left\{b\leq x<d+b : x^2\equiv\delta\pmod d\right\}\,.$$
When the integer $\delta$ is odd, the congruence $x^2\equiv\delta\pmod 4$ has a solution only if $\delta\equiv 1\pmod 4$, which is not true by our assumptions. If $\delta$ is even, we do not have solutions of $x^2\equiv\delta\pmod 4$ because $\delta$ is square-free.\\
For primes $p>2$ and $(p,\delta)=1$,  we have $\rho(p)=\#\left\{0\leq x<p : x^2\equiv\delta\pmod p\right\}$ , so $\rho(p^k)=1+\Leg{\delta}{p}$ for $k\geq 1$. If $(p,\delta)>1$, clearly $\rho(p)=1$. If $x$ is a solution of $x^2\equiv\delta\pmod {p^2}$, then $p$ divides $x$, and $\delta\equiv 0\pmod {p^2}$, which contradicts with $\delta$ being square-free. Therefore $\rho(p^k)=0$ if $k\geq 2$. \\

For a multiplicative function $\lambda(n)$ we denote the Dirichlet series $D_\lambda(s):=\sum_{n=1}^\infty \lambda(n)/n^s$. By the Chinese Remainder Theorem $\rho(d)$ is multiplicative, but not completely multiplicative. Obviously by definition (\ref{def1}) $\rho(d)\leq d$, so the Dirichlet series $D_{\rho}(s)$ is absolutely convergent for $\text{Re}(s)>2$. Therefore for $\text{Re}(s)>2$ we can write
\begin{align*}
D_\rho(s)=\sum_{n\geq 1}\frac{\rho(n)}{n^s}&=\prod_{p}\left(1+\frac{\rho(p)}{p^s}+\frac{\rho(p^2)}{p^{2s}}+\dots\right)\\
&=\left(1+\frac{\rho(2)}{2^s}\right)\prod_{p>2}\left(1+\frac{\rho(p)}{p^s}+\frac{\rho(p^2)}{p^{2s}}+\dots\right)\\
&=\left(1+2^{-s}\right)\prod_{\substack{p>2\\p\mid \delta}}\left(1+p^{-s}\right)\prod_{\substack{p>2\\(p,\delta)=1}}\left(1+\left(1+\Leg{\delta}{p}\right)\left(\frac{1}{p^s}+\frac{1}{p^{2s}}+\dots\right)\right)\\
&=\left(1+2^{-s}\right)\prod_{\substack{p>2\\p\mid\delta}}\left(1+p^{-s}\right)\prod_{\substack{p>2\\\Leg{\delta}{p}=1}}\left(1+2\left(\frac{1}{p^s}+\frac{1}{p^{2s}}+\dots\right)\right)\\
&=\left(1+2^{-s}\right)\prod_{\substack{p>2\\p\mid \delta}}\left(1+p^{-s}\right)\prod_{\substack{p>2\\\Leg{\delta}{p}=1}}\left(-1+\frac 2{1-p^{-s}}\right)\\
&=\left(1+2^{-s}\right)\prod_{\substack{p>2\\p\mid \delta}}\left(1+p^{-s}\right)\prod_{\substack{p>2\\\Leg{\delta}{p}=1}}\frac {1+p^{-s}}{1-p^{-s}}.
\end{align*}
Using definition (\ref{def2}) we can write 
$$\prod_{p>2}\frac{1+p^{-s}}{1-\chi(p)p^{-s}}=\prod_{\substack{p>2\\p\mid \delta}}\left(1+p^{-s}\right)\prod_{\substack{p>2\\\Leg{\delta}{p}=1}}\frac{1+p^{-s}}{1-p^{-s}}\prod_{\substack{p>2\\\Leg{\delta}{p}=-1}}\frac{1+p^{-s}}{1+p^{-s}}\,.$$
The third product equals $1$, so we get
$$D_\rho(s)=\left(1+2^{-s}\right)\prod_{p>2}\frac{1+p^{-s}}{1-\chi(p)p^{-s}}=\prod_{p}\left(1+p^{-s}\right)\prod_{p}\frac{1}{1-\chi(p)p^{-s}}=D_{\mu^2}(s)D_\chi(s)\,.$$

Then the coefficients of the Dirichlet series satisfy the identity
\begin{equation*}
\rho(d)=\sum_{lm=d}\mu^2(l)\chi(m)\,.\qedhere
\end{equation*}
\end{proof}

%---------------------------------------------------------------------
%--------  Lemma 2, character sum inequality

For any positive integer $N$ we denote 
\begin{equation}\label{eqX} X(N):=\sum_{1\leq n\leq N}\chi(n)\,.
\end{equation} 
We will need an explicit upper bound for the character sum $|X(N)|$. There are lots of works on such P\'olya-Vinegradov inequalities, aiming to reduce the upper bound, e.g. the papers of Qiu \cite{qiu} and Pomerance \cite{pomerance}. It is a question of taste which one to choose. We will apply the estimate of Qiu since its minor terms are somewhat easier.
\begin{lem}\label{lem2} Let $\delta$ be square-free, $\delta\not\equiv 1\pmod 4$, and consider the Dirichlet character $\chi$ defined in (\ref{def2}). For any $N\geq 1$ we have
$$\left|\sum_{n=1}^N \chi(n)\right|<\varkappa\,,$$
where $\varkappa=g(4|\delta|)$
%\[\varkappa=\varkappa(\delta)=\left\{\begin{array}{ll}
%g(|\delta|)& \text{, if } \delta\equiv 0\pmod 2\, ;\\
%g(4|\delta|)&\text{, if } \delta\equiv 2,3\pmod 4\,,
%\end{array}\right.
%\]
and $g(q):=\frac 1 2\sqrt{q}\log{q}+1.2\sqrt{q}$.
\end{lem}
\begin{proof} By the Theorem of Qiu \cite{qiu} for a primitive Dirichlet character $\chi$ modulo $q$ we have the inequality
$$\left|\sum_{n=1}^N \chi(n)\right|<\frac 4 {\pi^2}\sqrt{q}\log q+0.38\sqrt{q} + 0.608/\sqrt{q}+0.116 (N,q)^2/q^{\frac 3 2}\,.$$
Trivially $(N,q)^2\leq q^2$ and we can further bound from above the latter expression 
$$\left|\sum_{n=1}^N \chi(n)\right|<0.406\sqrt{q}\log q+0.496\sqrt{q} + 0.608/\sqrt{q}\,.$$
The expression on the right-hand side suggests to introduce the function $K(x):=0.812x\log x+0.496 x + 0.608/x$. By a simple calculation we can check that for $x\geq 1$ we have $K(x)<x\log x+1.2x$. Then 
\begin{equation}\label{eq:char2}\left|\sum_{n=1}^N \chi(n)\right|<K(\sqrt{q})<\sqrt{q}\log(\sqrt q)+1.2\sqrt{q}=\frac 1 2\sqrt{q}\log q+1.2\sqrt q=g(q)\,.
\end{equation}
Now we return to our character $\chi$ defined in (\ref{def2}).  We notice that we can write 
\[\chi(n)=\Leg{4\delta}{n}\,,\]
where $\Leg{.}{.}$ is the Kronecker symbol. Since $\delta\equiv 2,3\pmod 4$ is square-free, $4\delta$ is a fundamental discriminant. Therefore $\chi(n)$ is a primitive character of conductor $q=4|\delta|$. Now the statement of the Lemma follows from (\ref{eq:char2}).
\end{proof}

%------------------------------------------------------------------------

Let $x\geq 1$ be a real number. Using Lemma \ref{lem1} and Lemma \ref{lem2} we get

\begin{equation}\label{sum1}\sum_{d\leq x}\rho(d)=\sum_{lm\leq x}\mu^2(l)\chi(m)=\sum_{l\leq x}\mu^2(l)\sum_{m\leq x/l}\chi(m)
\leq \varkappa\sum_{l\leq x}\mu^2(l)\leq \varkappa x\,.
\end{equation}

Now returning to (\ref{eq1}) we see that we need to estimate the sums $\sum_{d\leq x}\rho(d)/d$, for which we use again Lemma \ref{lem1}:
\begin{equation}\label{eq2}\sum_{d\leq x}\frac{\rho(d)}d =\sum_{d\leq x}\sum_{lm=d}\frac{\mu^2(l)\chi(m)}{lm}=\sum_{l\leq x}\frac{\mu^2(l)}l\sum_{m\leq x/l}\frac{\chi(m)}m\,.
\end{equation}
Consider the sum $ \sum_{m\leq x}\chi(m)/m$ for a positive real $x\geq 1$. By Abel's summation formula we have
\begin{equation}\label{Abel}\Sigma :=\sum_{m\leq x}\frac{\chi(m)}m=\frac{X(x)}x-\int_1^x X(u)\left(\frac 1 u\right)'du=\frac{X(x)}x+\int_1^x \frac{X(u)}{u^2}du\,.\end{equation}
If $x\leq \varkappa$, the trivial bound $\left|\sum_{m\leq x}\chi(m)\right|\leq x$ is better than the universal bound provided by Lemma \ref{lem2}. Indeed, in that case from (\ref{Abel}) we obtain
$$\Sigma \leq  \frac x x + \int_1^x \frac u {u^2}du=1+\log x\leq 1+\log\varkappa\,.$$
If $x>\varkappa$, for (\ref{Abel}) we can write
\begin{eqnarray*}\Sigma &\leq &\frac \varkappa x +\int_1^\varkappa \frac{u}{u^2}du+\int_\varkappa^x\frac{\varkappa}{u^2}du\\
&\leq & 1+\log\varkappa+\varkappa\int_1^x \frac{du}{u^2}= 1+\log\varkappa-\frac \varkappa x+1<\log\varkappa +2\,.
\end{eqnarray*}

We conclude that for any $x\geq 1$
\begin{equation}\label{kappa}\sum_{m\leq x}\frac{\chi(m)}{m}<\log\varkappa +2
\end{equation}
and then (\ref{eq2}) transforms into
\begin{equation}\label{eq2_2}\sum_{d\leq x}\frac{\rho(d)}d\leq \left(\log\varkappa+2\right)\sum_{l\leq x}\frac{\mu^2(l)}l\,.
\end{equation}
For the last sum we apply an explicit upper bound due to Ramar\'e (Lemma 3.4 in \cite{ramare}) :
\begin{lem}(Ramar\'e, \cite{ramare})\label{lem3} Let $x\geq 1$ be a real number. We have
$$\sum_{n\leq x}\frac{\mu^2(n)}n\leq \frac 6{\pi^2}\log x+1.166\,.$$
\end{lem}
Applying this lemma in (\ref{eq2_2}) we get
\begin{equation}\label{sum2} \sum_{d\leq x}\frac{\rho(d)}d<\left(\log\varkappa+2\right)\left(0.608 \cdot\log x+1.166\right)\,.
\end{equation}
We plug the inequalities (\ref{sum1}) and (\ref{sum2}), with $x=\sqrt{f(N)}$, into (\ref{eq1}):
\begin{equation}\label{eq:sqrtF}\sum_{n\leq N}\tau(f(n))\leq 2N\left(\log\varkappa+2\right)\left(0.608\cdot \log \left(\sqrt{f(N)}\right)+1.166\right)+2\varkappa\sqrt{f(N)}\,.
\end{equation}

Now notice that $f(n)=n^2+2bn+c\leq n^2+2|b|n+|c|\leq (1+2|b|+|c|)n^2$ for $n\geq 1$. Then $\sqrt{f(N)}\leq\xi N$, where $\xi=\xi(b,c)=\sqrt{1+2|b|+|c|}$. Thus $\log\left(\sqrt{f(N)}\right)\leq \log\xi+\log N$.\\

Let $A$ be the least positive integer such that $A\geq\max\left(|b|,|c|^{1/2}\right)$. Another way to bound from above $\sqrt{f(N)}$ is by using $f(n)=n^2+2bn+c\leq n^2+2|b|n+|c|\leq n^2+2An+A^2=(n+A)^2$. Then $\sqrt{f(N)}\leq N+A$.\\

We apply these two bounds to transform further (\ref{eq:sqrtF}).
\begin{eqnarray*}\sum_{n\leq N}\tau(f(n))&\leq & 2N(\log\varkappa+2)\left(0.608\cdot \log N+0.608\cdot\log\xi+1.166\right)+2\varkappa(N+A)\\
&=& 1.216\left(\log\varkappa+2\right) N\log N+2\left(\varkappa+(\log\varkappa+2)\left(0.608\cdot\log\xi+1.166\right)\right)N+2\varkappa A\\
&=& C_1 N\log N + C_2 N + C_3\,,
\end{eqnarray*}
with the constants $C_1, C_2, C_3$ defined in (\ref{defC}). 
%We notice that for $|\delta|=2$, and only in this case from the set of discriminants we consider, we have $\varkappa<\log\varkappa +2$. Then, if in (\ref{Abel}) we use only Lemma \ref{lem2} to bound from above both $X(x)$ and $X(u)$, we obtain
%$$\sum_{m\leq x}\frac{\chi(m)}{m}<\varkappa\,.$$
%In our argument we use the latter inequality instead of (\ref{kappa}) and we obtain the constants $C_1, C_2, C_3$ defined in (\ref{defC2}), which are slightly smaller than the ones from definition (\ref{defC}), if $\varkappa<\log\varkappa +2$. 
This proves Theorem \ref{thm1}.

\section{Proof of Theorem \ref{thm2}}

If we apply Theorem \ref{thm1} for the polynomial $f(n)=n^2+1$, we obtain the bound
$$\sum_{n\leq N}\tau(n^2+1)<4.051\cdot N\log N+16.8 N+7.58\,.$$ 
We can do better if we notice that in (\ref{def2}) we actually deal with the odd Dirichlet character modulo $4$:
\begin{center}
$\chi(n)=\begin{cases}
1,&\text{if }n\equiv1\pmod 4 ;\\
-1,&\text{if }n\equiv3\pmod 4 ;\\
0,&\text{otherwise .}
\end{cases}$
\end{center}
In this case the character sum $X(N)$ defined in (\ref{eqX}) can take only values $0$ or $1$, so we do not need to use Lemma \ref{lem2}. We can replace the expressions $\varkappa$ in (\ref{sum1}) and $\log\varkappa +2$ in (\ref{sum2}) simply by $1$. Something more, the summation in (\ref{eq1}) over $d\leq\sqrt{N^2+1}$ is actually over $d\leq N$. Therefore Theorem \ref{thm2} follows from plugging the estimates (\ref{sum1}) and (\ref{sum2}) into  (\ref{eq1}), with $x=N$, and $1$ instead of $\varkappa$ and $\log\varkappa +2$.

\begin{rem} In the estimate (\ref{sum1}) we used the trivial bound $\sum_{l\leq N}\mu^2(l)\leq N$, but we can do slightly better for larger values of $N$. First we can use again Lemma 3.4 (Ramar\'e, \cite{ramare}) which says that for $N\geq 1000$ the constant $1.166$ from its statement can be substituted by $1.048$. Another result of Ramar\'e (Lemma 3.1,  \cite{ramare}) says that for $N\geq 1700$ we have  $\sum_{l\leq N}\mu^2(l)\leq 0.62\cdot N$. After a simple computer check for the cases $1000 \leq N<1700$ we see that this holds also for any positive integer in this range. Therefore for any $N\geq 1000$ we have slightly smaller upper bound:
$$\sum_{n\leq N}\tau(n^2+1)\leq 2N(0.608\cdot \log N+1.048)+2\cdot 0.62\cdot N=1.216 \cdot N\log N+3.336\cdot N\,.$$
\end{rem}

\section{Some examples} Using McKee's theorem from \cite{mckee} we can compute numerically the constant $\lambda(\delta)$ from the asymptotic formula
$$\sum_{n\leq N}\tau(f(n))\sim\lambda(\delta) N\log N\,,$$
where $f(n)=n^2+2bn+c$ and $\delta=b^2-c<0$. Let $(C_1, C_2, C_3)$ be the triple of constants from Theorem \ref{thm1}, such that
$$\sum_{n\leq N}\tau(f(n))<C_1 N\log N+C_2 N+C_3\,.$$ 
With this notation for $f(n)=n^2+1$ we have $\lambda(-1)\sim 0.955$ and $(C_1, C_2, C_3)\sim (4.051, 16.8, 7.58)$, whereas Theorem \ref{thm2} improves this to $(C_1, C_2, C_3)\sim (1.216, 4.332, 0)$. In general for large $|\delta|$ we have $C_1\sim\log{|\delta|}$, which is not too far from the coefficient in McKee's formula. By the class number formula one can see that $\lambda(\delta)$ is close to the value of the corresponding Dirichlet $L$-function at $1$. \\
  
More examples of the explicit upper bounds for few more polynomials is given in the following table.

\[ 
\def\arraystretch{1.7}
\arraycolsep=2.5pt
\begin{array}{| l | l | l | l | }
    \hline
    f(n) & \delta& \lambda(\delta) &\left (C_1,C_2,C_3\right) \\ \hline
    n^2+1 &-1 & 0.955 & (1.216, 4.332, 0) \\ \hline
    n^2+10n+27 & -2 &1.351 & (4.68, 30.15, 76.02)\\    \hline
    n^2+4n+10 & -6 & 1.56 & (5.7, 46, 110) \\    \hline
    n^2+52n+706 & -30 &  1.395 & (6.9, 115, 2126)\\ \hline
    n^2+10n-26 & 51 & - & (7.4, 138, 662) \\ \hline
%\caption{Examples}
  \end{array}
\]
\vspace{0.2cm}

The (easy) code for the performed computations can be found in \cite{Lcode}. It can be used to estimate explicitly the divisor sum over any other quadratic polynomial $f(n)$ which satisfies the conditions of Theorem \ref{thm1}. 

\subsection*{Acknowledgments.} The author thanks Christian Elsholtz for suggesting this problem. Both his and Andr\'as Bir\'o's comments on earlier versions of this note are appreciated. This work is partially supported by Hungarian Scientific Research Fund (OTKA) grant no. K104183.

\end{document}